\documentclass[a4paper, 12pt]{article}

\usepackage{amssymb, amsmath}
\usepackage{amsthm}
\usepackage{graphicx}
\usepackage{color}
\usepackage[ansinew]{inputenc}
\usepackage{indentfirst}

\theoremstyle{plain} 		\newtheorem{proposition}{Proposition}[section]
\theoremstyle{plain}		\newtheorem{propdef}[proposition]{Proposition-Definition} 
\theoremstyle{definition}	\newtheorem{defi}[proposition]{Definition}
\theoremstyle{plain}		\newtheorem{coro}[proposition]{Corollary}
\theoremstyle{plain} 		\newtheorem{lemma}[proposition]{Lemma}
\theoremstyle{definition} 	\newtheorem{example}[proposition]{Example}
\theoremstyle{remark} 		\newtheorem*{remarks}{Remarks}
\theoremstyle{remark} 		\newtheorem*{remark}{Remark}

\newcommand{\PRmn}{\mathbb{P}^m(\mathbb{F}_n)}
\newcommand{\PRtwon}{\mathbb{P}^2(\mathbb{F}_n)}
\newcommand{\Fn}{\mathbb{F}_n}

\newcommand{\PR}[2]{\mathbb{P}^#1(\mathbb{F}_#2)}
\newcommand{\F}[1]{\mathbb{F}_#1}
\newcommand{\ZZ}[1]{\mathbb{Z}\slash#1\mathbb{Z}}

\title{ Wada Dessins associated with Finite Projective Spaces and Frobenius Compatibility}
\author{Cristina Sarti\\
Mathematisches Seminar, Goethe Universität \\
Postfach 111932
D-60054 Frankfurt a.M., Germany\\
\textsf{sarti@math.uni-frankfurt.de}}

\begin{document}

\maketitle

\begin{abstract}
\textit{Dessins d'enfants} (hypermaps) are useful to describe algebraic properties of 
the Riemann surfaces they are embedded in.
In general, it is not easy to describe algebraic properties of
the surface of the embedding starting from the combinatorial
properties of an embedded dessin. However, this task becomes easier
if the dessin has a large automorphism group.\par 
In this paper we consider a special type of dessins, so-called
\textit{Wada dessins}. Their underlying graph illustrates the
incidence structure of finite projective spaces $\PR{m}{n}$.
Usually, the automorphism group of these dessins is a
cyclic \textit{Singer group} $\Sigma_\ell$ permuting transitively the 
vertices.
However, in some cases, a second group of automorphisms
$\Phi_f$ exists. It is a cyclic group generated by the 
\textit{Frobenius automorphism}.
We show under what conditions $\Phi_f$ is a group of automorphisms
acting freely on the edges of the considered dessins.


\vskip 3pt
\noindent \textit{\small Math.\ Subj.\ Class.: 05C10, 30F10, 05B10, 05B25, 51E20, 51D20}
\end{abstract}

\section{Introducing dessins d'enfants}

The term \textit{dessins d'enfants} was first used by Grothendieck (1984)
to refer to objects which are very simple but important to describe
Riemann surfaces.
Dessins d'enfants ${\cal D}$ can be defined as hypermaps on compact orientable surfaces.
A hypermap in its \textbf{Walsh representation} \cite{Walsh75} is a bipartite graph drawn without crossings
on a surface $X$ and cutting the surface into simply connected cells (faces).\par
For a vertex of the graph we call \textbf{valency} the number of
incident edges. For a cell, the valency is the number of edges on its boundary.
This number is always even for hypermaps. Edges bordering a cell from both sides have to be counted 
twice (see \cite{StWo01} and \cite{Lan04}).
A characteristic property of a dessin is its {signature} $\left ( p,q,r \right )$, where
\begin{enumerate}
\item
$p$ is the least common multiple of all valencies of the white vertices, 
\item
$q$ is the least common multiple of all valencies of the black vertices,
\item
$2r$ is the least common multiple of all face valencies.
\end{enumerate}
A dessin is called \textbf{uniform} if all white vertices have the same valency $p$, 
all black vertices have the same valency $q$ and all cells have the same valency
$2r$. Riemann surfaces with embedded dessins are algebraic curves of 
special type.
According to Bely\u{\i}'s theorem (see \cite{Belyi80}) a surface $X$ admits a 
model over the field
$\overline{\mathbb{Q}}$ if and only if there exists a 
meromorphic function $\beta$
\[
\beta: \; X \rightarrow \mathbb{P}^1(\mathbb{C})
\]
ramified above at most three points. Without loss of generality, these can 
be identified with $\{0,1, \infty\}$. 
The preimages of $0$ and $1$ are, respectively, the set of white 
and black vertices of the bipartite graph on $X$. The preimages of $\infty$
correspond to points within the faces. Each face contains exactly
one point which is commonly called \textbf{face center} (see e.g. \cite{Lan04}).
This means that we have a dessin on a surface if and only if we may 
describe it with equations whose coefficients are in the field $\overline{\mathbb{Q}}$.\par
In general, it is not easy to relate the combinatorial properties of a dessin to 
algebraic properties of the surface, such as defining equations or the
moduli field, but the task becomes easier if the embedded dessin has a large 
automorphism group. In the best case, this group acts transitively on
the edges of the dessin (for recent results see,
for instance, \cite{JoStWo09}) and we say that the dessin is not only uniform but 
even \textbf{regular}.\vskip 2ex\par
In this paper, we consider dessins whose underlying graph describes the incidence structure of
points and hyperplanes of projective spaces $\PR{m}{n}$ (see Sections \ref{sec_proj_sp}
and \ref{subsec_dessins}). 
The possibility to construct
this kind of dessins was first studied by Streit and Wolfart \cite{StWo01} for
projective planes $\PR{2}{n}$. Starting from some of their results, we examine 
here special kinds of uniform dessins called \textbf{Wada dessins} (see Section 
\ref{subsec_Wada_dessins}).
The cells of these dessins have all the same valency $2 \ell$
and there always exists a cyclic group $\Sigma_\ell$ transitively permuting the edges of
type $\circ$---$\bullet$
and of type $\bullet$---$\circ$ on their boundaries (see \cite{StWo01} and 
\cite{SartiPhD10}). We are interested in determining when the full automorphism group
contains other groups of automorphisms beside $\Sigma_\ell$.\par
We consider the \textbf{Frobenius automorphism} acting on points and on 
hyperplanes of projective spaces $\PR{m}{n}$ (see Section \ref{sec_frob_auto}) 
and we establish under what conditions it induces an 
automorphism of the associated Wada dessins (see Sections \ref{section_Frob_diff_set}
and \ref{subsec_Frob_compat}).
In general, it is not easy to predict the necessary restrictions on the parameters
$m$ and $n$ such that these conditions are satisfied. In Section \ref{sec_nice_case}
we prove that if $n$ and $m+1$ are primes the problem can be solved.\par 
Since in the literature automorphisms of dessins are defined in slightly different ways, 
we point out that here we consider only orientation-preserving automorphisms.
%
\section{Finite projective spaces}\label{sec_proj_sp}

Let $\mathbb{F}_n^{m+1}$ be the vector space over the finite field $\F{n}$, $n=p^e$ a prime power.
We define the finite projective space $\PR{m}{n}$ as the vector space 
$\mathbb{F}_n^{m+1}\backslash\{0\}$ factorized by the multiplicative group
$\mathbb{F}_n^*$ of order $n-1$:
\begin{equation}\label{eq_defproj}
\PRmn:=(\Fn^{m+1}\backslash\{0\})\slash\Fn^* \, .
\end{equation}
The integer $n$ is called \textbf{order} of $\PRmn$ (\cite{Beutelsp92}, \cite{Demb97}).\par
The number of points of $\PRmn$ is given by
\begin{equation}\label{eq_points_on hyperp}
\ell:=|\PRmn|=\frac{|\Fn^{m+1}\backslash\{0\}|}{|\Fn^*|}=\frac{n^{m+1}-1}{n-1}\, .
\end{equation}
By duality, the integer $\ell$ is also the number of hyperplanes of $\PRmn$.\par
\textbf{Hyperplanes} are subspaces $\mathbb{U}^{m-1}(\Fn)$ of dimension $m-1$ in $\PRmn$, 
thus they correspond to
subspaces of dimension $m$ in $\F{n}^{{m+1}}$. Similarly to 
equation (\ref{eq_points_on hyperp}), we compute the number 
of points on each of them as:
\begin{equation}\label{eq_integer_q}
q:=|\mathbb{U}^{m-1}(\Fn)|=\frac{|\Fn^{m}\backslash\{0\}|}{|\Fn^*|}=\frac{n^{m}-1}{n-1} \, .
\end{equation}
By duality, the integer $q$ is also the number of hyperplanes through each point.\par
The  finite field $\F{{n^{{m+1}}}}$ may be considered as a vector space
over the field $\F{n}$. 
A well known fact is that the multiplicative group
of a finite field is cyclic. Therefore the $m$-dimensional projective
space over $\F{n}$ may be identified with the quotient
\begin{equation}\label{eq:proj_sp_quot}
\PRmn\cong\mathbb{F}_{n^{m+1}}^*\slash \Fn^* \, ,
\end{equation}
which, being considered as a group, is also cyclic of order $\ell$.
Let $g$ be a generator of this group.\par
Due to the identification above we make correspond points
$P_b$ and, by duality, hyperplanes $h_w$ of $\PR{m}{n}$ to powers
of the generator $g$:
\begin{eqnarray}\label{eq_points_hyper_generators}
g^b \leftrightarrow P_b, \qquad g^w \leftrightarrow h_w \qquad b,w \in \mathbb{Z}\slash \ell \mathbb{Z} \, .
\end{eqnarray}
In this way, we obtain a numbering of points and of 
hyperplanes by the integers $b,w \in \{0, \cdots, \ell-1\}$.
Point and hyperplane numberings are closely related to each other and 
are not arbitrary as we will explain more in detail 
in Section \ref{subsec_dessins}.\par
Among the projective linear transformations of $\PR{m}{n}$
we find cyclic projectivities permuting in
a single cycle, i.e. transitively, the set of
points. By duality they also transitively permutes
the set of hyperplanes.
These projectivities generate so-called \textbf{Singer groups}
(see e.g. \cite{Hirsch97}).
Due to (\ref{eq:proj_sp_quot}), it is easy to prove that 
the group $\Sigma_\ell\cong\mathbb{F}_{n^{m+1}}^*\slash \Fn^*$ 
--with $\ell$ defined as in (\ref{eq_points_on hyperp})--
is cyclic and it acts transitively permuting the set of points 
and the set of hyperplanes of $\PR{m}{n}$. 
Hence, we formulate the following
\begin{propdef}\label{singer_group}
Every $\PRmn$ admits a point (hyperplane) transitive cyclic
group of automorphisms. This group is a \textbf{Singer group}.
\end{propdef}
This proposition was first proved by Singer (\cite{Singer38}) for projective planes 
$\mathbb{P}^2(\mathbb{F}_n)$. Now it is widely known (\cite{Baumert71}, \cite{Demb97}, \cite{Hirsch97})
that it also holds for spaces of higher dimension. For an exhaustive proof see 
\cite[Chapter II, Section 7 ]{Huppert79}. The existence of a Singer
group cyclically permuting points and hyperplanes justifies the fact that
projective spaces $\PR{m}{n}$ are commonly called 
\textbf{cyclic projective spaces}.\vskip 2ex
Let $\gamma$ be a generating element of $\Sigma_\ell$. Thus, due to
(\ref{eq_points_hyper_generators}) and to (\ref{eq:proj_sp_quot}), 
the action of each element \mbox{$\gamma^a \in \Sigma_\ell$},
\mbox{$a \in \ZZ{\ell}$},
on the points $P_b$ and on the hyperplanes $h_w$ is naturally expressed by:
\begin{align}\label{eq_singer_action}
\gamma^a:& P_b \longmapsto P_{b+a}\; , \nonumber\\
				& h_a \longmapsto h_{w+a} \, .
\end{align}
%
\subsection{Constructing dessins}\label{subsec_dessins}

In order to construct dessins associated with a projective space $\PR{m}{n}$, we need to construct
the corresponding bipartite graph. We first introduce the conventions in Table \ref{table_conventions}.\\
\begin{table}[h]
	\begin{center}
			\begin{tabular}{|l||c|}\hline
				point&black vertex $\bullet$\\\hline
				hyperplane&white vertex $\circ$\\\hline
				incidence&joining edge ---\\\hline
			\end{tabular}
		\end{center}
		\caption{\textit{Conventions}}\label{table_conventions}
\end{table}\\
\par Incidence between a point $P_b$ and an hyperplane $h_w$ is illustrated
by the bipartite graph through a joining edge between a black vertex 
$b$ and a white vertex $w$.
Recall that points $P_b$ and hyperplanes $h_w$ are numbered with integers
$b,w \in \{0, \cdots, \ell-1\}$ given by the exponents of a generator $g$ of
$\mathbb{F}_{n^{m+1}}^*\slash \Fn^*$ (see Section \ref{sec_proj_sp}, 
Relation (\ref{eq_points_hyper_generators}) ).
This numbering is not arbitrary. We consider the $q$ points on a line.
Thanks to the existence of the Singer group defined in Proposition 
\ref{singer_group}, Singer could show (see \cite{Singer38}) 
that for projective planes 
the integers resulting from differences of 
a line index with the indices of the incident points
form a difference set.
Difference sets are defined in the following way:
\begin{defi}[\cite{Baumert71}]
A $(v,k,\lambda)$-difference set $D=\{d_1, \cdots, d_k\}$ is a collection of $k$ residues modulo $v$, such 
that for any residue $\alpha \not \equiv 0 \mod v$ the congruence
\begin{eqnarray}\label{diffset:eq_relat}
 d_i - d_j \equiv \alpha \mod v
\end{eqnarray}
has exactly $\lambda$ solution pairs $(d_i, d_j)$ with $d_i$ and $d_j$ in $D$.
\end{defi}
In particular, sets $(D+s) \mod v$ with $s \in \ZZ{v}$ are also difference
sets and we call them \textbf{shifts} of $D$. If $\widehat{D}\equiv (t \cdot D + s) \mod v$
with $s \in \ZZ{v}$, $t \in (\ZZ{v})^*$ then $\widehat{D}$ and $D$ are said to be 
\textbf{equivalent}.\par
For projective planes $v$ is equal to $\ell$, which is the total number of points 
(lines), and $k$ is equal to $q$, which is the number of points on a line and, by duality, of lines
through a point. Singer's construction tells us that a point $P_b$ and a line $h_w$ are
incident if and only if
\begin{equation}\label{incidence_relation}
b-w \equiv d_i \mod \ell \, , \quad i \in \ZZ{q} \, ,
\end{equation}
where the $q$ elements $d_i$ are the
elements of a difference set $D$.
It is a well known fact that Singer's construction can be extended to projective spaces
of higher dimension (see, for instance, \cite{Hirsch97}), 
i.e. differences of point indices with the index of the common incident
hyperplane build an \mbox{$(\ell, q, \lambda)$ - difference set}. This difference
set determines a numbering of points depending on hyperplane
numbering and viceversa (see Relation (\ref{incidence_relation})).
Nevertheless, the difference set is not unique
since we may have several difference sets with parameters $(\ell, q, \lambda)$ 
(see \cite{Baumert71}) which are equivalent to $D$ or not. We thus 
fix one difference set $D$ and
one ordering of its elements, unique up to cyclic permutations.\par
According to (\ref{incidence_relation}) and identifying each point and each
hyperplane with its index $b$ or $w$, we choose the local incidence pattern 
given in Figure \ref{fig_local}
\footnote{We remark here that this is not the only possible incidence pattern
we may choose. According to the fixed ordering of the elements
of the difference set,
the white vertices incident with a black vertex are 
ordered anticlockwise and the black vertices incident with a white one are 
ordered clockwise. We have chosen these orderings since we are interested
in special uniform dessins called \textbf{Wada dessins} wich we will introduce later on in
this section and which are the main topic of study of this paper. \par
The choice 
of different orderings is also possible and may give rise to dessins 
which are not of Wada type (see \cite[Chapter 6]{SartiPhD10}).}.\par
\begin{figure}[!ht]
	\centering
	\includegraphics[width=.7\linewidth]{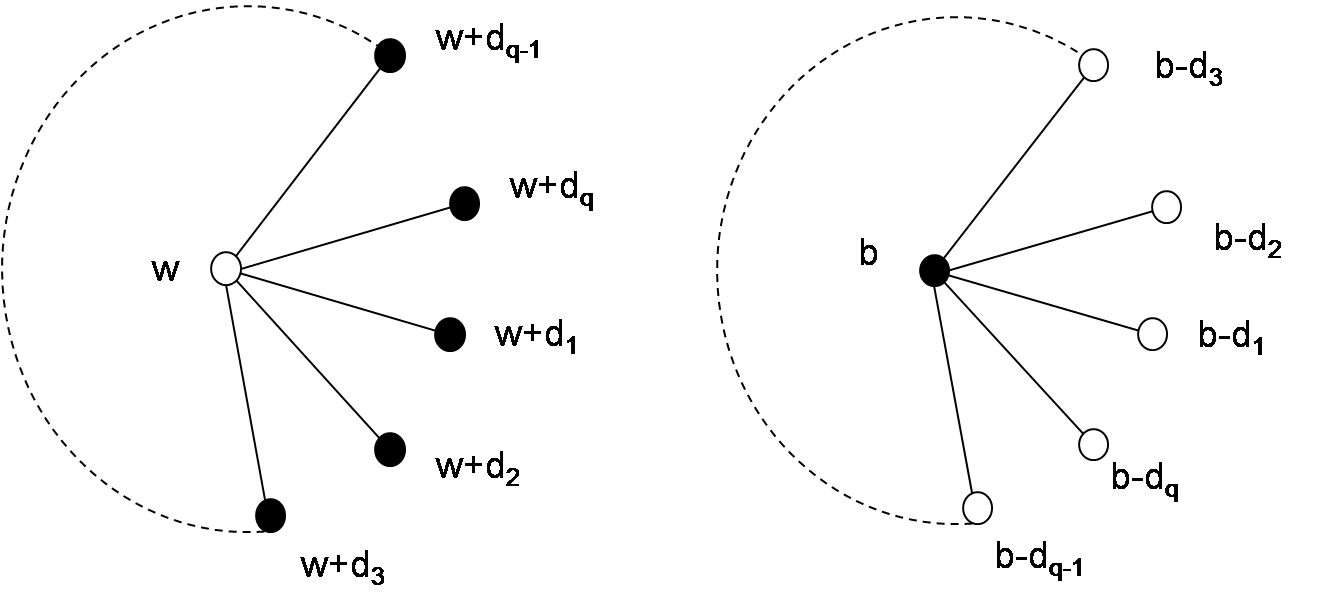}
	\caption{\textit{Local incidence pattern.}}
	\label{fig_local}
\end{figure}
We construct the embedding of the bipartite graph starting 
with a white vertex $w$ and with an incident edge $(w, w+d_i)$. 
Thus according to the local incidence pattern of the 
black and of the white vertices we have chosen, the next incident edge 
going clockwise around a cell 
is $(w+d_i, w+d_i-d_{i+1})$ followed by 
$(w+d_i-d_{i+1}, w+2d_i-d_{i+1})$. Repeating this procedure, we obtain a cell 
boundary with the sequence of edges given in Figure \ref{fig_cell}.\par
%
\begin{figure}[!ht]
	\centering
	\includegraphics[width=.8\linewidth]{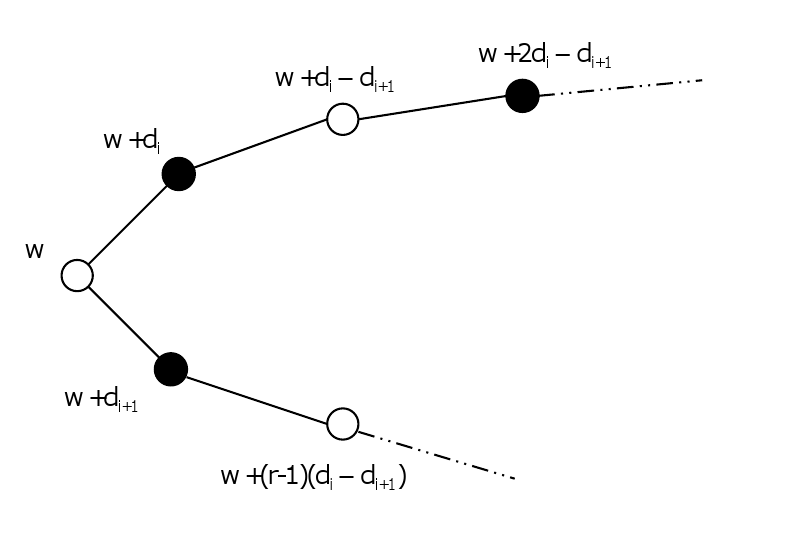}
	\caption{\textit{Construction of a cell incident with
	a white vertex $w$. After $2r$ steps we reach the starting
	edge $(w,w+d_i)$.}}
	\label{fig_cell}
\end{figure}
%
The construction is applied to successive white vertices 
until $q \cdot \ell$ different edges of type $\circ$---$\bullet$ are 
constructed. Different means here that every two edges differ at least
in one of the two indices $b$ or $w$. The construction terminates since
all $q \cdot \ell$
incidences between points and hyperplanes are represented by 
the bipartite graph
\footnote{For reasons of duality between hyperplanes (white vertices)
and points (black vertices) we may also carry out the construction
starting with edges incident with black vertices. In this case
we come to an end when $q \cdot \ell$ different edges of type 
$\bullet$---$\circ$ are constructed.}.
For each cell, depending on the value of differences $d_{i}-d_{i+1}$
we reach the starting edge after $2r$ steps,
with $r:=\frac{\ell}{gcd((d_{i}-d_{i+1}),\ell)}$. Thus we obtain cells with valencies
$2r$ (see Figure \ref{fig_cell}), where $r$ does not necessarily
have the same value for all cells.

\subsection{Wada dessins}\label{subsec_Wada_dessins}

We now consider a special case of the above construction.\par
Suppose that we may find at least one cyclic ordering of the elements of $D$ such that all differences 
$d_i - d_{i+1}$ are prime to $\ell$. In this case, all cells of the 
dessin we construct have the same valency $2\ell$ as it is easy to check. 
The dessin is therefore uniform 
with signature $( q,q,\ell ) $ and with $q$ cells.
Such dessins have the following nice property. On the boundary of each cell each white and each black
vertex with a given index occurs precisely once. If two white vertices had the same index, then
according to the incidence pattern in Figure \ref{fig_local} and to the 
construction described above we should have
\[
w + \alpha \cdot (d_i - d_{i+1}) \equiv w + \beta \cdot (d_i -d_{i+1}) \; , \quad \alpha, \beta \in \ZZ{\ell} \; ,
\]
but this is possible only for $\alpha \equiv \beta \mod \ell$ since differences 
$(d_i - d_{i+1})$ are prime to $\ell$. In a similar way we may prove
that it is not possible to have two black vertices
with the same index on the same cell boundary.\par
This property was first described by Streit and Wolfart
\cite{StWo01} for bipartite graphs of projective planes $\PRtwon$ 
and is called
\textbf{Wada property}. The choice of the name \textit{Wada} goes back
to the theory of dynamical system and of the \textit{Lakes of Wada}. It is,
in fact, possible to divide the euclidean plane in three regions such
that all the points on the boundary of one region are also on the boundary
of the other two (see e.g. \cite[Chapter 4]{Buskes97}). To emphasize the analogy 
between this phenomenon and the one observed for the cells of their dessins,
Streit and Wolfart called them \textbf{Wada dessins}.\par
More in general, the Wada property may also be 
observed for embeddings of bipartite graphs
associated with projective spaces of higher dimension. By construction, the Wada dessins we obtain are always uniform.
Moreover, due to the action of the cyclic group $\Sigma_\ell$ 
on the elements of projective spaces (see Section 
\ref{sec_proj_sp}), these dessins have the property that at least
the group $\Sigma_\ell$ is a group of automorphisms. 
This group acts transitively on the black and on the white vertices permuting them cyclically
on the cell boundaries. This action induces a transitive action
on the set of edges of type $\circ$---$\bullet$ and on the set of edges of
type $\bullet$---$\circ$ belonging to the boundary of each cell 
(see \cite[Chapter 5]{SartiPhD10} for more details).\par 
We conclude by giving the following definition of \textbf{Wada compatibility} for 
orderings of elements
of difference sets $D$ associated with projective spaces:
\begin{defi}\label{defi_wada}
We call orderings of the $q$ elements of a difference set $D$ associated 
with a projective space 
$\PR{m}{n}$ {\bf Wada compatible} if differences $\mod \ell$ of consecutive 
elements $d_i, d_{i+1}$ of 
$D$ are prime to $\ell$, i.e.:
\begin{eqnarray}
gcd\left((d_{i}-d_{i+1}), \ell\right)=1 \quad \forall i \in \ZZ{q} \, .
\end{eqnarray}
\end{defi}
\begin{remark}

As we have remarked above the difference set $D$ with the chosen element ordering
is not unique. Multiplying $D$ with integers $t \in \left(\ZZ{\ell}\right)^*$ and
shifting it with integers $s \in \ZZ{\ell}$ we obtain new difference sets equivalent
to $D$. If we do not change the element ordering of the original difference set
or if we only permute it cyclically,
the Wada dessins we construct are isomorphic to each other. However,
for non-cyclic permutations and for difference sets
with the same parameters $(\ell, q, \lambda)$ but non-equivalent to $D$,
even if the chosen ordering is Wada compatible, the dessins we construct
are in general not isomorphic to each other. We may, in fact, obtain a different embedding
of the graph into an orientable surface and a different Bely\u{\i} function.\par 
For more general and open questions
about if and how many non-equivalent difference sets with the same parameter set exist,
see e.g. \cite{Baumert71}.
\end{remark}
%
\section{The Frobenius Automorphism}\label{sec_frob_auto}

We consider again the identification of the projective space $\PR{m}{n}$, $n=p^e$,
with the quotient of multiplicative cyclic groups:  
\[
\PRmn\cong\mathbb{F}_{n^{m+1}}^*\slash \Fn^* \, .
\]
Since $\mathbb F_{n^{m+1}}$ is
a finite field, we know that there exists an automorphism 
$\sigma$, the {\bf Frobenius automorphism},
acting on the elements $a \in \F{{n^{m+1}}}$ in the following way:
\begin{eqnarray*}
	\begin{array}{clcl}
		\sigma:  & \mathbb{F}_{n^{m+1}} & \longrightarrow & \mathbb{F}_{n^{m+1}}  \\[1ex]
		  & a & \longmapsto &   a^{p} \, .
	\end{array}
\end{eqnarray*}
The Frobenius automorphism generates the Galois group Gal$(\mathbb F_{n^{m+1}}/\mathbb F_p)\cong \Phi_f$ as 
a cyclic group of order $f=e \cdot (m+1)$.\par
As we have seen in Section \ref{sec_proj_sp}, we may identify points $P_b$ and 
hyperplanes $h_w$
of $\PR{m}{n}$ with powers of a generator of $\mathbb{F}_{n^{m+1}}^*\slash \Fn^*$ (see 
Relation (\ref{eq_points_hyper_generators})).
Thus each \mbox{$\sigma^u \in \Phi_f$, $u \in \ZZ{f}$} acts on $P_b$ and $h_w$ as:
\begin{eqnarray}\label{eq_action_Cf_indices}
	\begin{array}{cccc}
		\sigma^u: & P_b & \longmapsto &   P_{b{p^u}}\\
			  & h_w & \longmapsto & h_{w{p^u}} \, .
	\end{array}
\end{eqnarray}
The action of $\Phi_f$ subdivides the set of points ${\cal P}$ and 
the set of hyperplanes ${\cal H}$
into orbits with lengths $\varphi \in \mathbb{N}$, $\varphi \mid f$. 
Each set always contains at least one orbit of length one, i.e. the orbit of the point $P_0 \in {\cal P}$ 
and the orbit of the hyperplane $h_0 \in {\cal H}$. In fact,
\begin{eqnarray*}
	\begin{array}{ccccr}
		\sigma^u: & P_0 & \longmapsto &   P_0  & \\
			& h_0 & \longmapsto & h_0 & \forall u \in \mathbb{Z}/f\mathbb{Z}\, .
	\end{array}
\end{eqnarray*}
%
\subsection{Frobenius difference sets}\label{section_Frob_diff_set}

Not every automorphism of projective spaces $\PR{m}{n}$ leads to
an automorphism of associated dessins d'enfants.
Whether the Frobenius automorphism induces an automorphism of associated 
dessins or not, depends, basically, on the difference set we choose for 
the construction. First of all, let us look at an example.
\begin{example}\label{example_P_F2_1}
We consider the projective space $\PR{4}{2}$ with parameters $\ell=31$ and $q=15$.
For the construction of associated dessins we use the difference set $D$ and
the shift $D'\equiv D-1 \mod 31$ with their elements ordered in the following way:
\begin{align*}
D &=\{1,3,15,2,6,30,4,12,29,8,24,27,16,17,23 \}\mod 31 \, ,\\[2ex]
D' &\equiv D - 1 \mod 31 \\
    &= \{0,2,14,1,5,29,3,11,28,7,23,26,15,16,22\} \mod 31 \, .
\end{align*}
Recalling the incidence pattern given in Figure \ref{fig_local}, 
we construct two dessins ${\cal D}$ and ${\cal D}'$. We are sure
that they have the Wada property, since $\ell=31$ is prime. The dessins have
signature $(15,15,31)$ and 15 cells.
\begin{figure}[!ht]
\includegraphics[width=.5\linewidth]{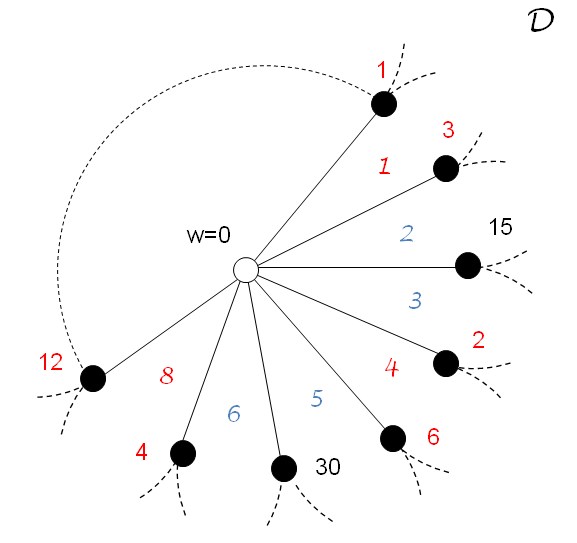}
\includegraphics[width=.5\linewidth]{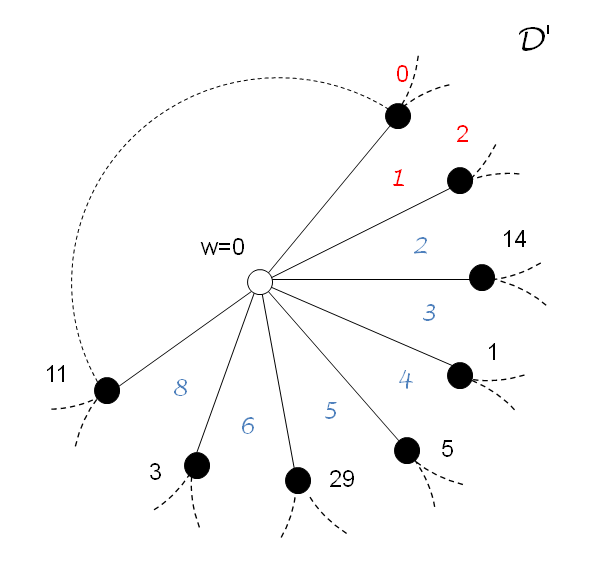}
\caption{\textit{Sketch of dessins $\cal D$ and ${\cal D}'$
associated with $\PR{4}{2}$. In red we emphasize the action of the Frobenius
automorphism $\sigma$ on the cells. The action of $\sigma$ turns out to be
an automorphism of the dessin $\cal D$ but not
of the dessin ${\cal D}'$ (see text for details).}}
\label{fig_dessinsP4_F2}
\end{figure}
Due to the action on points and on hyperplanes of $\PR{4}{2}$ (see
Relation (\ref{eq_action_Cf_indices})), the Frobenius automorphism $\sigma$ acts 
on the vertices of each dessin by a multiplication with the prime two:
\begin{align}\label{equ_frob_action_P4_F2}
\sigma: &\quad b \longmapsto b \cdot 2 \, , \nonumber \\
				&\quad w \longmapsto w \cdot 2 \, .
\end{align}
Lookig at Figure \ref{fig_dessinsP4_F2} and with an easy check of the action on
the vertices, we observe that $\sigma$ is an automorphism of
$\cal D$ but not of ${\cal D}'$. On ${\cal D}$ it rotates the cells 
around the fixed vertex $w=0$ --and by duality also around $b=0$--
by an angle $\omega = \frac{2 \pi}{5}$.\par
The reason for the different action is due to the different behaviour of the sets 
$D$ and $D'$ under the action of $\sigma$. According to Singer's construction, elements
of difference sets associated with projective spaces $\PR{m}{n}$ correspond to 
differences of indices of points and of incident hyperplanes. 
Thus, recalling the action of $\sigma$ on points and hyperplanes
(see Relation (\ref{eq_action_Cf_indices})), its action on the elements of $D$ and $D'$
is a multiplication with the integer two.
Under such multiplication the set $D$ is fixed up to
a cyclic permutation of its elements. On the contrary, the set
$D'$ is not fixed, as it is easy to check:
\begin{align*}
&\sigma(D)=\{2,6,30,4,12,29,8,24,27,16,17,23,1,3,15 \}\mod 31 \, ,\\[3ex]
&\sigma(D')=\{0,4,28,2,10,27,6,22,25,14,15,21,30,1,13\} \mod 31 \, .
 \end{align*} 
\end{example}
\bigskip
We will analyse more closely the ordering of the elements of $D$ in the next section.\par
Here we formulate the following
\begin{propdef}\label{prop:frobenius_auto}
Let $\PR{m}{n} \cong \F{{n^{m+1}}}^* \slash \F{n}^*$ be a projective space, 
$n=p^e$, $p$ prime, containing $\ell$ points and $\ell$ hyperplanes
whose incidence can be described using difference sets.
The Frobenius automorphism determines at least one
difference set $D_f$ fixed under the action of 
\mbox{$\Phi_f\cong \textrm{Gal}(\mathbb F_{n^{m+1}}/\mathbb F_p)$}.
Multiplying $D_f$ with integers \mbox{$t \in (\ZZ{\ell})^*$},
we obtain further difference sets fixed under the action of $\Phi_f$.\par
We call $D_f$ and all difference sets $t \cdot D_f$ \textbf{Frobenius difference sets}.
\end{propdef}
\begin{proof}
As we have seen in Section \ref{sec_frob_auto}, $\Phi_f$ divides the set of 
points $\cal P$ and the set of
hyperplanes $\cal H$ of $\mathbb P^m(\mathbb F_n)$ into orbits with possibly different 
lengths. In particular, there is always at least one orbit of length one in each set. In fact, 
$\Phi_f$ always fixes the elements $P_0 \in {\cal P}$
and $h_0 \in {\cal H}$. Let us consider the hyperplane $h_0$. 
Since $\Phi_f$ is a group of automorphisms of $\PR{m}{n}$, each
$\sigma^u \in \Phi_f$ preserves incidence. If we now consider the set 
${\cal P}_0=\lbrace P_{i_0},P_{i_1},\cdots,P_{i_{q-1}}\rbrace$ of points on $h_0$, since
\begin{eqnarray*}
 \sigma^u: h_0 \longmapsto h_0 \quad \forall \sigma^u \in \Phi_f \, ,
\end{eqnarray*}
by incidence preservation we have as well:
\begin{eqnarray*}
 \sigma^u: {\cal P}_0 \longmapsto {\cal P}_0 \quad \forall \sigma^u \in \Phi_f\, .
\end{eqnarray*}
According to Singer's construction,
differences of point indices with the index of the common incident hyperplane
describe a difference set (see Section \ref{subsec_dessins}), thus the indices of
the points belonging to ${\cal P}_0$ form the Frobenius difference set $D_f$
we are looking for. It is easy to see that multiplying
$D_f$ with integers $t \in (\ZZ{\ell})^*$ we still obtain 
Frobenius difference sets. In fact, 
$\Phi_f$ acts on the elements of $D_f$ and of each $t \cdot D_f$ by 
multiplication with powers $p^u$, $u \in \ZZ{f}$, so we have:
\begin{eqnarray*}
p^u \cdot t \cdot D_f = t \cdot p^u \cdot D_f \equiv t \cdot D_f \mod \ell 
\quad \forall t \in (\ZZ{\ell})^* \, ,
\end{eqnarray*}
i.e. each $t \cdot D_f$ is fixed under the action of $\Phi_f$.
\end{proof}
\begin{remark}
The difference set $D_f$ and the sets $t \cdot D_f$ determined by 
$\Phi_f$ may not be unique. 
In fact, as we have already remarked in Sections 
\ref{subsec_dessins} and \ref{subsec_Wada_dessins}, other difference sets 
$D_f$ with parameters $(\ell,q,\lambda)$ may exist which are equivalent or 
non-equivalent to $D_f$ and are fixed by $\Phi_f$. 
The case of difference sets resulting from shifts of $D_f$ 
will be analysed more closely later on in this section.
\end{remark}
The above proof implies
\begin{coro}\label{coro_Cf_action_FrobAuto}
The cyclic group $\Phi_f$ acts on the elements of a Frobenius difference set $D_f$
by a multiplication with powers $p^u$, $u \in \ZZ{f}$ and we have
\[
 p^u D_f\equiv D_f \mod \ell \quad \forall u \in (\mathbb Z/f \mathbb Z) \, .
\]
\end{coro}

Due to the action of $\Phi_f$ on points and hyperplanes, 
from the above proof it also follows that $\Phi_f$ divides
the points on $h_0$ into orbits with lengths $\varphi$, $\varphi \in \mathbb{N}$, 
$\varphi \mid f$. As $D_f$ consists of the indices of
these points, this means that $\Phi_f$ also divides the elements of $D_f$ into orbits 
with lengths $\varphi$. Hence, we formulate the following
\begin{coro}\label{coro_Frobenius_orbits}
Under the action of $\Phi_f$, the elements of a Frobenius difference set $D_f$ are 
subdivided into orbits with lengths
$\varphi, \, \varphi \in \mathbb N, \, \varphi \mid f$. 
These orbits correspond to orbits of 
points in $\mathbb P^m(\mathbb F_n)$.
\end{coro}
\begin{example}\label{example_orbits_P4_F2}
In the above example for $\PR{4}{2}$ the Frobenius automorphism 
generates the Galois group $\Phi_5\cong\textrm{Gal}(\F{2^5}\slash \F{2})$. This group 
acts on the points $P_i$ and on the hyperplanes $h_i$
of $\PR{4}{2}$ by a multiplication with powers $2^u$, $u \in \ZZ{5}$:
\begin{eqnarray*}
 \begin{array}{lc@{\,\longmapsto\,}cr}
 \sigma^u \in \Phi_5:	& P_i & P_{i{2^u}} \\[2ex]
		& h_i & h_{i{2^u}}, & u\in (\mathbb Z/5\mathbb Z) \, .
 \end{array} 
\end{eqnarray*}
The elements of the difference set $D_5$ may be identified
with the indices of the 15 points on the fixed hyperplane $h_0$. 
The difference set $D_5$ is 
fixed under the action of $\sigma$ and therefore of $\Phi_5$. Its elements
are subdivided by $\Phi_5$ into three orbits of length five:
\[
\{1,2,4,8,16\} \, , \quad \{3,6,12,24,17\} \, , \quad \{15,30,29,27,23\} \, .
\]
\end{example}
\begin{example}\label{example_P3_F3_frobenius_difference_sets}
Let us consider the projective space 
$\mathbb P^3(\mathbb F_3)\cong\mathbb F_{3^4}^*/\mathbb F_{3}^*$
with $40$ points and, by duality, with $40$ hyperplanes. 
The Frobenius automorphism $\sigma$ generates the Galois group 
$\Phi_4 \cong \textrm{Gal}(\mathbb F_{3^4}/\mathbb F_3)$ which acts on the points $P_i$ and on
the hyperplanes $h_i$ as
\begin{eqnarray*}
 \begin{array}{lc@{\,\longmapsto\,}cr}
 \sigma^u \in \Phi_4:	& P_i & P_{i{3^u}}  \\[2ex]
		& h_i & h_{i{3^u}}, & u \in (\mathbb Z/4\mathbb Z) \, .
 \end{array} 
\end{eqnarray*}
The elements of the 
following $(40,13,4)$-difference set (\cite{Baumert71}) may be identified with
the indices of the 13 points on the hyperplane $h_0$ fixed by $\Phi_4$:
\[
D_4=\lbrace 21, 22, 23, 25, 26, 29, 34, 35, 38, 0, 5, 7, 15  \rbrace \mod 40 \, .
\]
The set $D_4$ is a Frobenius difference set and
it is easy to prove that multiplying it with powers $3^u$ we only 
have a permutation of its elements, i.e. we have
\[
3^u D_4 \equiv D_4 \mod 40 \quad \forall u \in (\mathbb Z/4\mathbb Z) \, .
\]
The cyclic group $\Phi_4$ subdivides the elements of $D_4$ into the following five orbits:
\[
\lbrace {21}, {23}, {29}, {7} \rbrace, \lbrace {22}, {26}, {34}, {38} \rbrace,
\lbrace {25}, {35}\rbrace, \lbrace {5}, {15}\rbrace, \lbrace 0 \rbrace \, .
\]
\end{example}
\bigskip
We have seen that difference sets which result from multiplications of Frobenius 
difference sets $D_f$ by elements $t \in (\ZZ{\ell})^*$ are still Frobenius
difference sets fixed by $\Phi_f$. Therefore, it is reasonable to ask
about shifts. Indeed, depending on the number of fixed 
points and of fixed hyperplanes
of $\Phi_f$ we can have more than one Frobenius difference set:
\begin{proposition}\label{prop:frobenius_dset_unicity}
Let $\PR{m}{n}\cong \F{{n^{m+1}}}^*\slash \F{n}^*$ be a projective space, $n=p^e$, 
$p$ prime. The projective space $\PR{m}{n}$ contains $\ell$ points and
$\ell$ hyperplanes whose incidence can be described using a Frobenius difference set $D_f$.
No shifts of $D_f$ are allowed 
if and only if only $h_0$ and equivalently only $P_0$
are fixed by $\Phi_f \cong \textrm{Gal}(\F{{n^{m+1}}}\slash \F{p})$. 
If there are more elements fixed, we will have shifts
of $D_f$ which are still Frobenius difference sets and:
\[
\# \textrm{Frobenius difference sets as shifts of } D_f = 
\# \textrm{ fixed hyperplanes (fixed points) of } \Phi_f \, .
\]
\end{proposition}
\begin{proof}
Let us assume that no shifts of $D_f$ are allowed but there is at least
another hyperplane $h_s$ fixed by $\Phi_f$.
This would mean that $h_0$ and $h_s$ share the same point set ${\cal P}_0$. 
Nevertheless, sharing the same point set means $h_0=h_s$.\par
Now we assume that only the hyperplane $h_0$ is fixed by $\Phi_f$, but 
there are two difference sets fixed:
the Frobenius difference set $D_f$ and a shift of it $D'_f \equiv (D_f + s) \mod \ell$. 
Both $D_f$ and $D_f'$ are defined  
up to multiplication with elements $t \in (\ZZ{\ell})^*$. According to Singer's construction 
if we identify the points on $h_0$ with the elements of $D_f$, then we may identify
the points on $h_s$ with the elements of $D'_f \equiv (D_f + s) \mod \ell$. As $D'_f$ is fixed under 
the action of $\Phi_f$ so should also $h_s$ be, but this would be a contradiction to the fact that $h_0$ is unique. 
It thus follows that no shifts of $D_f$ are allowed.\par
As we have already seen, the indices of
the points on every hyperplane $h_s$, \mbox{$s \in \ZZ{\ell}$} can be identified with 
the elements of a shift $D_f + s$.
If $\Phi_f$ fixes some of the hyperplanes $h_s$, it directly follows
that the indices of the points on each fixed $h_s$ describe Frobenius difference sets $D_f + s$ and we have:
\[
\# \textrm{Frobenius difference sets as shifts of } D_f = 
\# \textrm{ fixed hyperplanes (fixed points) of } \Phi_f \, .
\]
\end{proof}
\begin{remark}
\label{remark:frobenius_dset_shifts}
From the proof of Proposition \ref{prop:frobenius_dset_unicity} follows that, 
once we know a Frobenius difference set $D_f$, it is easy to determine 
other equivalent difference sets fixed by $\Phi_f$.
We only need to know which of the hyperplanes $h_s$ is fixed by $\Phi_f$.
The corresponding difference set is then $D_f+s$.
\end{remark}
\begin{example}
For the projective space $\PR{4}{2}$ there are only one hyperplane
$h_0$ and one point $P_0$ fixed under the action of $\Phi_5$. It follows that the
Frobenius difference set 
\[
D_5=\lbrace 1,3,15,2,6,30,4,12,29,8,24,27,16,17,23 \rbrace \mod 31
\]
is unique up to multiplication with elements $t \in (\ZZ{31})^*$.
\end{example}
\begin{example}
For the projective space $\PR{3}{3}$ there are two hyperplanes and, by duality, two points
fixed by $\Phi_4$. The fixed hyperplanes are $h_0$ and $h_{20}$
and the fixed points are $P_{0}$ and $P_{20}$. As we have seen in Example
\ref{example_P3_F3_frobenius_difference_sets}, we may choose
\[
D_4=\lbrace 21, 22, 23, 25, 26, 29, 34, 35, 38, 0, 5, 7, 15  \rbrace \mod 40
\]
as the Frobenius difference set corresponding to 
the set of points on $h_{0}$. We, therefore, have that
\[ 
D'_4\equiv D_4+20 \mod 40=\lbrace 1,2,3,5,6,9,14,15,18,20,25,27,35 \rbrace \mod 40
\]
is the Frobenius difference set corresponding to the set of points on 
$h_{20}$ and it is easy to prove that 
$3^u (D_4 + 20)\equiv(D_4 + 20) \mod 40$ for all $u \in (\ZZ{4})$.
\end{example}
%
\subsection{Frobenius compatibility}\label{subsec_Frob_compat}

The existence of Frobenius difference sets
is a necessary but still not a sufficient condition
for the construction of dessins having
the Frobenius automorphism as a dessin automorphism. We consider the following
example:
\begin{example}\label{example_P_F2_2}
In addition to the dessins $\cal D$ and ${\cal D}'$ we have considered in Example 
\ref{example_P_F2_1}
we construct a third dessin ${\cal D}''$ associated with the projective space
$\PR{4}{2}$. For the construction we use
the difference set $D''$ whose elements are the elements of
$D$ but with a different ordering\footnote{We use here again the notation $D$ for the difference set 
$D_5$ in order to be consistent with the notation used in Example \ref{example_P_F2_1}.} 
:
\begin{eqnarray*}
D=\{1,3,15,2,6,30,4,12,29,8,24,27,16,17,23 \}\mod 31 \, ,\\[2ex]
D'' =\{1,2,3,4,6,8,12,15,16,17,23,24,27,29,30\}\mod 31 \, .
\end{eqnarray*}
According to the incidence pattern given in Figure \ref{fig_local}
we may construct the Wada dessin ${\cal D}''$ sketched
in Figure \ref{fig_dessinsP4_F2_3rddess}.
\begin{figure}
\includegraphics[width=.5\linewidth]{P4F2_Incidence_D1.png}
\includegraphics[width=.5\linewidth]{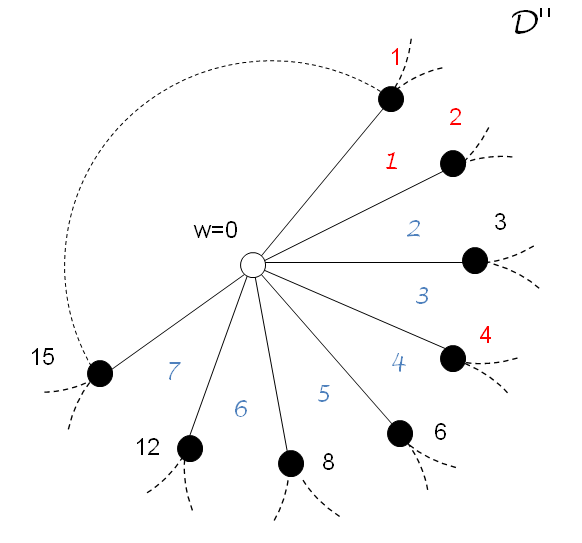}
\caption{\textit{Sketch of dessins $\cal D$ and ${\cal D}''$
associated with $\PR{4}{2}$. As in Figure \ref{fig_dessinsP4_F2}
we emphasize in red the action of the Frobenius automorphism on
the cells (see text for more details).}}
\label{fig_dessinsP4_F2_3rddess}
\end{figure}
In Example \ref{example_P_F2_1} we have seen that $\sigma$ is an automorphism
of $\cal D$ acting with a rotation of the cells around the white vertex
$w=0$. However, for ${\cal D}''$ we observe that it is not an automorphism.
The different behaviour of
the Frobenius automorphism on the dessins ${\cal D}$ and ${\cal D}''$
depends on the ordering of the elements of the difference sets.
We consider orbits of elements under the action of the cyclic group $\Phi_5$ generated by
$\sigma$. As we have seen in Example \ref{example_orbits_P4_F2}
the elements of $D$ are subdivided into three orbits of length five: 
\[
\{1,2,4,8,16\} \, , \quad \{3,6,12,24,17\} \, , \quad \{15,30,29,27,23\} \, .
\]
Thus, we are able to construct five blocks: 
the first block only contains the first elements of each orbit, 
the second block the second elements of each orbit, etc.
The Frobenius automorphism maps the elements
of each block onto the elements of the next block.
The element ordering of $D$ corresponds to this subdivision
into blocks, the element ordering of $D''$ does not.
\end{example}
We call the ordering of the elements of $D$ given in the above example \textbf{Frobenius compatible}.\par
According to the action of $\Phi_f$ on the elements of a Frobenius difference set $D_f$ 
(see Corollary \ref{coro_Cf_action_FrobAuto}), we give the 
following definition:
\begin{defi}\label{defi_frob_compat}
Consider the following cyclic orderings of the $q$ elements of a Frobenius difference set $D_f$:
\begin{align}\label{eq_frob_compat}
D_f=&\{d_1, \cdots , d_{k},\; p^j d_1, \cdots, p^j d_{k},\;\cdots\cdots \cdots,
		p^{(f-1)j}d_1 , \cdots, p^{(f-1)j} d_{k}\}\; ,\nonumber\\
		 &\textrm{ for some } j \in \left(\mathbb{Z}\slash f \mathbb{Z}\right)^*, \;\frac{q}{f}=k\; .
\end{align}
We call such cyclic orderings {\bf Frobenius compatible} orderings.
\end{defi}
Since Wada compatibility is also necessary for the construction of Wada dessins
(see Definition \ref{defi_wada}) we formulate moreover
\begin{proposition}\label{prop_wadaord}
The elements of a Frobenius difference set $D_f$ ordered in a Frobenius 
compatible way are also ordered in a Wada compatible way 
iff differences of consecutive elements belonging to 
the subset $\{d_1, \cdots , d_{k}\}$ are prime to $\ell$
\begin{equation}\label{eq_difffrob_0}
 gcd \left(d_{i}- d_{i+1}, \ell\right)=1 \quad \forall i \in \{1, \cdots, (k-1)\}
\end{equation}
and iff
\begin{equation}\label{eq_difffrob_1}
gcd \left(d_{k}-p^{j}d_1,\ell\right)=1 \, .
\end{equation}
\end{proposition}
\begin{proof}
If the elements of $D_f$ ordered in a Frobenius compatible way also satisfy the Wada condition,
then conditions (\ref{eq_difffrob_0}) and (\ref{eq_difffrob_1}) necessarily hold.\par
If condition (\ref{eq_difffrob_0}) holds then we also have
\begin{equation}\label{eq_difffrob_0_general}
gcd\left(p^j(d_{i}- d_{i+1}), \ell\right)=1 \quad \forall j \in \left(\ZZ{f}\right)^* \; .
\end{equation}
In fact, since $\ell=\frac{n^{m+1}-1}{n-1}=n^m+n^{m-1}+\cdots + 1$ with $n=p^e$ 
we have $gcd(\ell, p) = 1$.
This means that all differences of consecutive elements belonging to the subset 
$\{p^j d_1 , \cdots , p^j d_{k}\}$ are prime to $\ell$.
Of course, this may be extended to each other subset
$\{p^{h\cdot j} d_1 , \cdots , p^{h\cdot j} d_{k}\}$, $h \in \ZZ{f}$.\par
Condition (\ref{eq_difffrob_1}) is necessary to make sure that when passing from one subset to the next
the Wada condition is also satisfied. If
(\ref{eq_difffrob_1}) holds then for the same reasons as in
(\ref{eq_difffrob_0_general}) we also have:
\begin{eqnarray*}
gcd((p^{h\cdot j}d_{k} - p^{(h+1)\cdot j}d_1), \ell)=1 \quad \forall h \in \ZZ{f} \, .
\end{eqnarray*}
\end{proof}
We proceed considering Wada dessins constructed with Frobenius difference sets whose 
elements are ordered in a Frobenius compatible way. 
For these dessins the cyclic group $\Phi_f$ is a group of automorphisms.
We examine this fact more in detail in the next sections.
%
\section{A 'nice' case}\label{sec_nice_case}

In Section \ref{section_Frob_diff_set} we have seen that for each projective 
space $\PR{m}{n}$ there always exists at least one Frobenius difference set $D_f$. Unfortunately,
we cannot always order its elements in a Frobenius compatible way. In fact, it may happen
that the cyclic group $\Phi_f$ divides the elements of $D_f$ into orbits with different lenghts. 
In Example \ref{example_P3_F3_frobenius_difference_sets}
we have seen that the cyclic group $\Phi_4$ divides the elements of the
difference set $D_4$ associated with $\PR{3}{3}$ into orbits with lenghts four, two
and one.
In general, it is not easy to determine the necessary conditions for the parameters $m$
and $n$ of $\PR{m}{n}$ so that we can predict the existence of
Frobenius compatible orderings. Here we consider a 'nice' case and we see that if
the order of the cyclic group $\Phi_f$ is prime, 
under some conditions, Frobenius compatible orderings exist.\par
We first formulate two lemmas:
\begin{lemma}\label{lemma_fdivq}
Let $f$ be the order of the cyclic group $\Phi_f\cong Gal(\F{{n^{m+1}}} \slash \F{p})$ 
generated by the Frobenius automorphism acting on a projective space $\PRmn$ 
with \mbox{$n=p^e, e \in \mathbb{N}\backslash\{0\}$}. Let $f$ be prime, thus we have:
\begin{eqnarray}
n=p \; , \quad f=m+1 \; , \quad \Phi_f\cong Gal(\F{{p^{m+1}}} \slash \F{p}) \; .
\end{eqnarray}
If $p \neq m+1$ and $p\not\equiv 1 \mod (m+1)$, 
then $f$ divides the valency $q$ of the white and of the black vertices.
\end{lemma}
\begin{proof}
The proof follows from Fermat's little theorem.\par
For $n=p$ the integer $q$ is given by (recall Relation
(\ref{eq_integer_q})):
\[
q=\frac{p^m - 1}{p -1} \, .
\]
The integer $f=m+1$ divides $q$ if
\[
p^m \equiv 1 \mod (m+1) \, ,
\]
and this is true due to Fermat's little theorem since we have chosen $p\neq m+1$. 
Now we need the denominator not to 'destroy' the divisibility property. 
In fact, we have to choose
\[
p \not\equiv 1 \mod (m+1)\, ,
\]
since for $p \equiv 1 \mod (m+1)$ we obtain:
\[
q=p^{m-1}+ \cdots + 1 \equiv m \mod (m+1)
\]
and $gcd(m,(m+1))=1$.
\end{proof} 
\begin{lemma}\label{lemma_fdiffset}
We consider the Frobenius difference set $D_f$ fixed by the cyclic group $\Phi_f$.
Under the conditions of Lemma \ref{lemma_fdivq} on $f$ and $n$,
no shifts of $D_f$ are fixed by $\Phi_f$ and 
the following properties hold:
\begin{enumerate}
\item 
$0 \not\in D_f$ and
\item
all $\Phi_f$-orbits of elements $d_i \in D_f$ have length $f$.
\end{enumerate}
\end{lemma}
\begin{proof}
As we have seen in Corollary \ref{coro_Cf_action_FrobAuto}, the group 
$\Phi_f$ acts on the elements of $D_f$ by
multiplication with powers of $p$ and it divides the elements into orbits whose lengths are divisors of $f$. 
Since $f$ is prime, 
we will only have orbits of length $1$ or of length $f$. Suppose we have at least one orbit of length 1.
Since $f$ divides $q$ (see Lemma \ref{lemma_fdivq} above), we necessarily have at least $f-1$ other
orbits of length $1$. 
Let $0 \in D_f$ and let $\{0\}$ be an orbit of length one.\par
Recalling some ideas of Baumert \cite{Baumert71} about projective planes, we first show that $gcd(p-1,\ell)=1$.\par 
Dividing $\ell$ by $p-1$ we obtain:
\[
\ell=(p-1)(p^{m-1}+ \cdots + (m-1)p + m) +(m+1)\, .
\]
We know
\[
gcd(p-1,m+1)=1
\]
due to $p\not\equiv 1 \mod (m+1)$ and to $m+1$ being prime. 
From the Euclidean algorithm it therefore follows that
\[
gcd(p-1,\ell)=1 \, .
\]
Let us suppose that there exists another Frobenius difference set
$D_f' \equiv (D_f + s) \mod \ell$ fixed by $\Phi_f$. For the integer $s \in \ZZ{\ell}$ 
we have:
\begin{equation}\label{eq_FrobDiff_unique}
p\cdot s \equiv s \mod \ell \quad \textrm{ i.e. } \quad (p-1)\cdot s \equiv 0 \mod \ell \, .
\end{equation}
Since $gcd(p-1,\ell)=1$, equation (\ref{eq_FrobDiff_unique}) can be satisfied
only for $s \equiv 0 \mod \ell$. Thus no shifts of
$D_f$ are allowed.\par
Now, if $(f-1)$ elements $d_i \in D_f$ with $d_i \not\equiv 0 \mod \ell$ are fixed by $p$, then they 
should satisfy the congruence relation:
\[
p\cdot d_i \equiv d_i \mod \ell \; \Longrightarrow \; (p-1)\cdot d_i \equiv 0 \mod \ell \, .
\]
Again, since $gcd(p-1,\ell)=1$, the above congruence can be satisfied only for 
\mbox{$d_i \equiv 0 \mod \ell$}.
This means that $\{0\}$ is the only possible orbit of length 1 under the action of $p$. 
Nevertheless, we must exclude it since in this case $f$ would not divide $q$. \par
It follows that $0 \not\in D_f$ and that $D_f$ is decomposed only into $\Phi_f$-orbits of length $f$.
\end{proof}
Thanks to the above lemma we may choose cyclic orderings of 
the elements of $D_f$ as we have defined them in \ref{defi_frob_compat}. We have
\begin{coro}
Under the conditions of Lemma \ref{lemma_fdivq}, the elements of a Frobenius difference set 
$D_f$ fixed by the cyclic group $\Phi_f$ may be ordered in a Frobenius compatible way. This ordering 
is fixed under the action of $\Phi_f$ up to cyclic permutations.
\end{coro}
Using the Frobenius difference set $D_f$ with Frobenius compatible element orderings 
we construct dessins associated with projective spaces.
If the dessins have the Wada property, then the cyclic group $\Phi_f$ 
is a group of automorphisms.
%
\begin{figure}[!ht]
	\includegraphics[width=.5\linewidth]{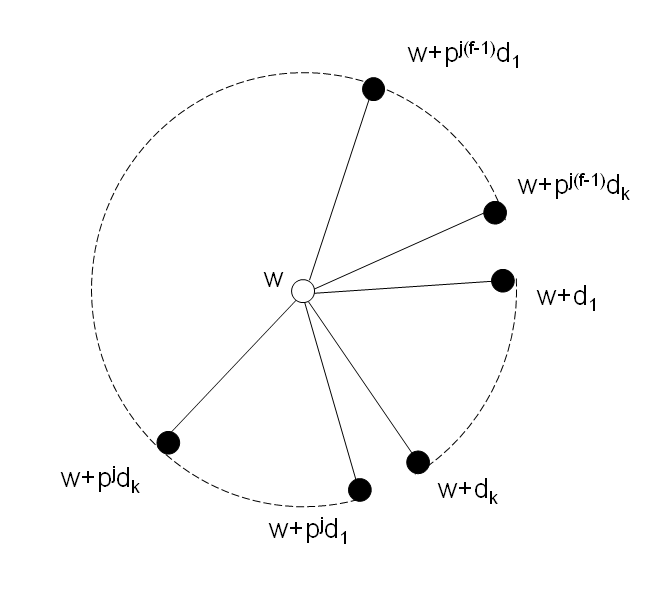}
	\includegraphics[width=.5\linewidth]{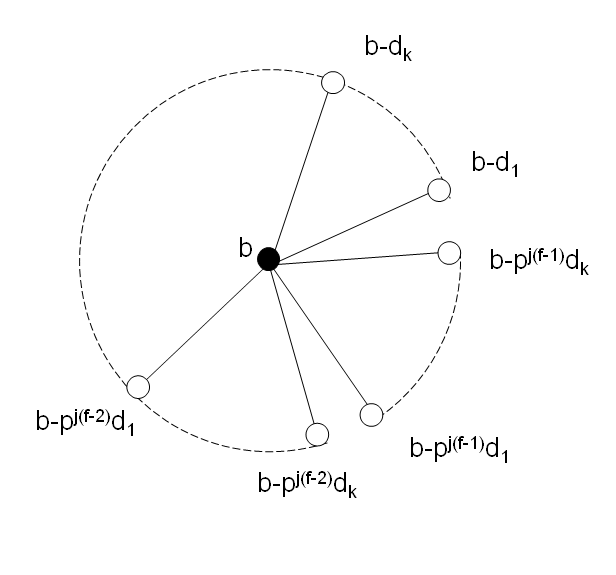}
	\caption{\textit{Local incidence patterns with Frobenius compatible
ordering of the elements of $D_f$.}}
\label{fig_FrobCompPattern}
\end{figure}
%
\begin{proposition}\label{prop_frobauto}
Let $\Phi_f$ be the cyclic group generated by the Frobenius automorphism
acting on a finite projective space $\PRmn$, $n=p^e$. 
Let $f$ be prime and $p \neq m+1$, $p \not\equiv 1 \mod (m+1)$. Let $D_f$ be a 
Frobenius difference set fixed by 
$\Phi_f$ whose elements are ordered in a Frobenius compatible way. If the elements
of $D_f$ are also ordered in a Wada compatible way so that we can construct
a $(q,q,\ell)$-Wada dessin $\cal D$, then $\Phi_f$ is a group of automorphisms
of $\cal D$ acting freely on the edges and rotating the set of cells 
around the vertices $b=w=0$ fixed by $\Phi_f$.
\end{proposition}
\begin{proof}
We suppose that at least one Frobenius compatible ordering of the elements of $D_f$
considered up to cyclic permutations is also Wada compatible and we construct a 
$(q,q,\ell)$-Wada dessin $\cal D$. Since $f|q$ 
(see Lemma \ref{lemma_fdivq} above), the group $\Phi_f$ has a 
suitable size to be a group of automorphisms of $\cal D$ acting on the $q$ cells of $\cal D$.
The cyclic group $\Phi_f$ is generated by the Frobenius automorphism $\sigma$.
We consider the action of $\sigma$ on the points and hyperplanes of 
$\PRmn$ and therefore on the vertices of $\cal D$. As we have seen in Section \ref{sec_frob_auto}
it acts by multiplication of the indices with the prime $p$:
\begin{align*}
\sigma: &P_b \longmapsto P_{bp}\\ &h_w \longmapsto h_{wp}\, .
\end{align*}
According to the notation introduced in Definition \ref{defi_frob_compat}, we write
the edges of \, $\cal D$ \, as \mbox{$e_{\nu,i}^w= (w,w+p^{\nu}d_i)$} \; ($\circ$ --- $\bullet$) \;
and \; \mbox{$e_{\nu,i}^b=(b,b-p^\nu d_i)$} \; ($\bullet$ --- $\circ$)\; with
$\nu \in \ZZ{f}$, \mbox{$i \in \{ 1, \cdots , k\}$}.
The action of $\sigma$ on the edges is given by:
\begin{align*}
\sigma: & \; e_{\nu,i}^w=\left(w,w+p^\nu d_i\right)\longmapsto 
								p\cdot e_{\nu,i}^w=\left(p\cdot w, p(w+p^\nu d_i)\right)\\
	& \; e_{\nu,i}^b=\left(b,b-p^\nu d_i\right)\longmapsto 
									p\cdot e_{\nu, i}^b=\left(p\cdot b, p(b-p^\nu d_i)\right)\; .
\end{align*}
If $w$ and $b$ are not fixed under the action of $\Phi_f$, then none of the edges 
$e_{\nu,i}^w$ and $e_{\nu,i}^b$ is fixed by 
$\sigma$. If, on the contrary, $w$ and $b$ are fixed, then $\sigma(e_{\nu,i}^w) \neq e_{\nu,i}^w$ and 
$\sigma(e_{\nu,i}^b)\neq e_{\nu,i}^b$ only if $d_i$ is not fixed by $p$. Indeed, this is true,
otherwise we could not have chosen $D_f$ with a Frobenius 
compatible ordering of its elements. Thus 
$\Phi_f$ does not fix any of the edges and we say that its action is \textbf{free} on them.\par
We now consider the action of $\Phi_f$ on the cells around the vertices $w$ and $b$ of 
$\cal D$ which are fixed by $\Phi_f$. Since $\Phi_f$ only fixes the difference set $D_f$ 
(see Lemma \ref{lemma_fdiffset}), this means that it only
fixes the vertices $b=w=0$ of the dessin. In fact, according to Singer's construction
(see Section \ref{subsec_dessins}), elements of difference sets $D$ associated
with projective spaces correspond to indices of points on hyperplanes (and by duality of
hyperplanes through points).
As $D_f$ is the only difference set fixed by $\Phi_f$, only $h_0$
and $P_0$ are fixed by $\Phi_f$, from which it follows that $w=b=0$ are 
the only vertices fixed by $\Phi_f$.\par
Recall the incidence pattern of each vertex 
given in Figure \ref{fig_FrobCompPattern} where the elements of $D_f$ are 
ordered in a Frobenius compatible way. The incidence pattern is fixed
up to cyclic permutations of the elements of $D_f$.
The action of 
$\sigma$ on the cells ${\cal C}_c$, $c\in \mathbb{Z}\slash q\mathbb{Z}$ 
around $w=0$ results in a mapping of every cell to a following cell 
(see Figure \ref{fig_FrobCell}) such that:
\begin{align}\label{eq_cell_mapping}
&\sigma: {\cal C}_c \longmapsto {\cal C}_{c+mk \mod q},\nonumber\\ 
&c\in \mathbb{Z}\slash q\mathbb{Z}\; , 
m \in \mathbb{Z}\slash f\mathbb{Z}\; , k=\frac{q}{f}\; .
\end{align}
We obtain the following relation between powers of $p$ and the cells we run through:
\begin{center}
\begin{tabular}{r|l}
$p$-powers & cells\\
\hline\\
$p^0$ & ${\cal C}_c$\\
$p^1$ & ${\cal C}_{c+mk\mod q}$\\
$p^2$ & ${\cal C}_{c+2mk\mod q}$\\
$p^3$ & ${\cal C}_{c+3mk\mod q}$\\
$\vdots$ & $\vdots$\\
$p^{f-1}$ & ${\cal C}_{c+(f-1)mk\mod q}$
\end{tabular}
\end{center}
Evidently, $\sigma$ describes a rotation of the set of cells around 0. 
The order of the rotation is $f$.
\end{proof}
\begin{figure}[!h]
	\centering
	\includegraphics[width=.7\linewidth]{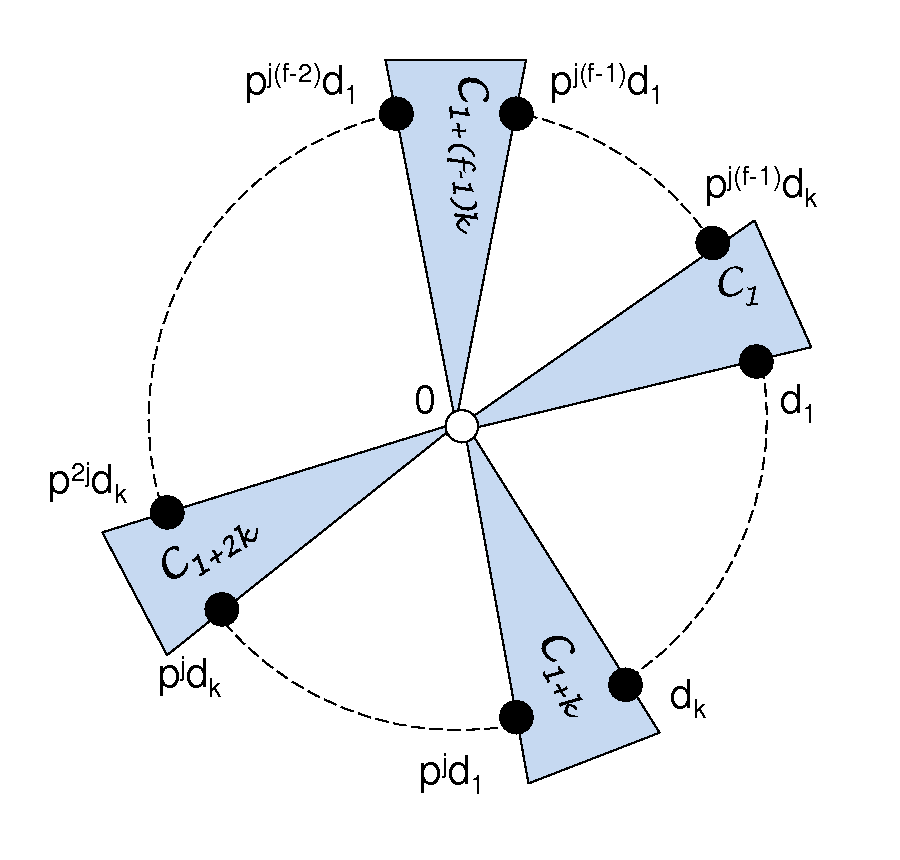}
	\caption{\textit{Local ordering of the cells around $w=0$ with a Frobenius difference set $D_f$. The elements of $D_f$ are ordered in a Frobenius compatible way.}}
	\label{fig_FrobCell}
\end{figure}
%
\begin{remarks}
\begin{enumerate}
\item 
A consequence of the above proof is that $\Phi_f$ acts freely not only on the edges but also
on the cells and it divides them into $k$ orbits
of length $f$.
\item
We consider the vertices fixed by $\Phi_f$ as fixed points $x$ on the target surface 
$X$ of the embedding.
Then, introducing local coordinates $z$ we may suppose that $z(x)=0$ and that 
$\Phi_f$ acts on a neighbourhood of $z=0$
by multiplication with powers of a root of unity $\zeta_f$:
\[
z \longmapsto \zeta_f^a z\; , \quad a \in \mathbb{Z}\slash f \mathbb{Z} \, .
\]
The root of unity $\zeta_f$ is called {\bf multiplier} of 
the automorphism $\sigma \in \Phi_f$ on $X$ (see \cite{StWo01}). 
\end{enumerate}
\end{remarks}
\begin{example}
For the projective space $\PR{4}{2}$, we have considered in several of the above examples, we have
$f=m+1=5$ and $p \neq 5$, $p \not\equiv 1 \mod 5$, thus the conditions 
of Proposition \ref{prop_frobauto}
are satisfied. We construct a $( 15,15,31 )$-Wada dessin with the Frobenius difference set:
\[
D_5 =\{1,3,15,2,6,30,4,12,29,8,24,27,16,17,23 \}\mod 31
\]
whose elements are ordered in a Frobenius compatible way (see Example \ref{example_P_F2_1}
and \ref{example_P_F2_2}). The cyclic group 
$\Phi_5\cong\textrm{Gal}(\F{{2^5}}\slash \F{2})$, generated
by the Frobenius automorphism, acts freely on 
the edges and on the cells. On the cells it acts with a rotation by the angle 
$\omega = \frac{2 \pi}{5}$
around the fixed vertices $b=w=0$. The cells are subdivided into three orbits of length five.\par
Other dessins, whose parameters satisfy the conditions of Proposition \ref{prop_frobauto},
are given in Table \ref{table_Frobenius_spaces}. For most of the projective spaces listed there
$\ell$ is prime, so we are sure we may construct Wada dessins regardless of the chosen
element orderings of the associated Frobenius difference set.\par
For the spaces $\PR{4}{3}$ and $\PR{{10}}{2}$ the integer $\ell$ is not prime, thus we have to check whether
Frobenius compatible orderings of elements are also Wada compatible. According
to Proposition \ref{prop_wadaord}, we only need to check differences of 
the first block of elements and the one difference at the
'transition' between the first and the second block. Other differences are
multiplications with powers of the prime $p$ for which we have $gcd(\ell,p)=1$.\par
For both spaces the prime factors of the integer $\ell$ 
are quite big. In fact, we have $\ell=121=11 \cdot 11$ and $\ell=2047=89 \cdot 23$, so
it is very likely to find Frobenius and Wada compatible orderings. 
For instance, for $\PR{4}{3}$ it is easy to check
that the ordering of the elements of the following Frobenius difference set is 
Frobenius and Wada compatible.
\begin{align*}
D_5=\{&1,4,7,11,13,34,25,67,3,12,21,33,39,102,75,80,9,36,63,99,117,64,104,119\\
	&27,108,68,55,109,71,70,115,81,82,83,44,85,92,89,103\} \mod 121 \, .
\end{align*}
In this case, the group generated by the Frobenius automorphism is cyclic of order five
and it divides the elements of $D_5$ into eight orbits of length five.
\begin{center}
\begin{table}
\begin{tabular}{c|c|c|c|c|c}
\multicolumn{6}{l}{}\\
\hline
 & $q$ & $\ell$ & $f$ & Wada & Frobenius\\
\hline
$\PR{2}{5}$ & $6$ & $31$ & 3 & $( 6,6,31 )$ & $\Phi_3\cong$Gal$(\F{{5^3}}\slash \F{5})$\\[2ex]
$\PR{4}{2}$ & $15$ & $31$ & 5 & $( 15,15,31 )$ & $\Phi_5\cong$Gal$(\F{{2^5}}\slash \F{2})$\\[2ex]
$\PR{4}{3}$ & $40$ & $121=11^2$ & 5 & $( 40,40,121 )$ & $\Phi_5\cong$Gal$(\F{{3^5}}\slash \F{3})$\\[2ex]
$\PR{4}{7}$ & $400$ & $2801$ & 5 & $( 400,400,2801 )$ & $\Phi_5\cong$Gal$(\F{{7^5}}\slash \F{7})$\\[2ex]
$\PR{6}{2}$ & $63$ & $127$   & 7 & $( 63,63,127 )$    & $\Phi_7\cong$Gal$(\F{{2^7}}\slash \F{2})$\\[2ex]
$\PR{6}{3}$ & $364$ & $1093$ & 7 & $( 364,364,1093 )$ & $\Phi_7\cong$Gal$(\F{{3^7}}\slash \F{3})$\\[2ex]
$\PR{6}{5}$ & $3906$ & $19531$ & 7 & $( 3906,3906,19531 )$ & $\Phi_7\cong$Gal$(\F{{5^7}}\slash \F{5})$\\[2ex]
$\PR{{10}}{2}$ & $1023$ & $2047=23\cdot 89$ & 11 & $( 1023,1023,2047 )$(?) & $\Phi_{11}\cong$Gal$(\F{{2^{11}}}\slash \F{2})$\\
\hline
\multicolumn{6}{l}{}\\
\multicolumn{6}{l}{(?) = only if there exists a Frobenius compatible ordering which is 
also Wada compatible.}
\end{tabular}
\caption{\textit{Some projective spaces whose parameters satisfy the conditions of 
Proposition \ref{prop_frobauto}.}}
\label{table_Frobenius_spaces}
\end{table}
\end{center}
\end{example}
%
\section{Concluding remarks}\label{section_concluding_remark}

We consider projective spaces $\PR{m}{n}$ with general $m$ and $n$, 
which do not necessarily satisfy 
the conditions of Proposition \ref{prop_frobauto}. In this case, 
it is more difficult to predict whether 
the cyclic group $\Phi_f$ or a subgroup $\Phi_g \subset \Phi_f$ is a group of automorphisms 
of associated Wada dessins.\par
Let $\Phi_g \subset \Phi_f$ be generated by a power $\sigma^s$ of the Frobenius 
automorphism, with \mbox{$s \in (\ZZ{f})$} and $gcd(s,f) \neq 1$. 
The action of $\Phi_g$ on the elements of the Frobenius difference set
we use to construct the associated dessin is a multiplication by
the integer $t=p^s$.\par
As we have seen in Sections \ref{section_Frob_diff_set} 
and \ref{subsec_Frob_compat}, for $\Phi_f$ to be a group of automorphisms
of the constructed dessin we need not only a Frobenius difference set $D_f$, 
but we also need the ordering of the elements of $D_f$ to
be compatible with the action 
of $\Phi_f$. Equivalent conditions are necessary if we need a subgroup
$\Phi_g \subset \Phi_f$ to be a group of automorphisms of the dessin.
This means that under the action of $\Phi_g$ the elements of a fixed difference set 
$D_g$ have to be subdivided into orbits of equal length.
This can be achieved if the following conditions hold:
\begin{enumerate}
\item
The order $g$ of $\Phi_g$ divides $q$, the number of elements of $D_g$;
\item\label{condition_2}
The integer $t$ satisfies $gcd(t-1,\ell)=1$.
\end{enumerate}
We need the first condition since it implies that $\Phi_g$ has a suitable size to subdivide
the elements of $D_g$ into orbits of the same length $g$.\par
The second condition makes sure that all orbits have the 
same length. In fact, having $gcd(t-1,\ell)=1$ means that the relation
\[
t \cdot d_i \equiv d_i \mod \ell \quad \forall d_i  \in D_g
\] 
is satisfied only for $d_i \equiv 0 \mod \ell$. But we exclude this possibility, otherwise
$D_g$ would contain the only orbit $\{0\}$ of length one and all other orbits would have length
$g$. This would be a contradiction to the 
first condition $g \mid q$, so we have $d_i \not\equiv 0 \mod \ell$ for all $d_i \in D_g$ 
and all orbits have length $g$.\par
Under these conditions, it is possible to order the elements of a difference set
fixed by $\Phi_g$ in a way compatible with its action, i.e. such that $\Phi_g$ acts
permuting the elements cyclically. If the ordering is also Wada compatible, then
$\Phi_g$ is a group of automorphisms of the constructed Wada dessin.
Similarly to the action of $\Phi_f$ (see Section \ref{sec_nice_case}) also
$\Phi_g$ acts on the cells by rotating them around the fixed vertices $w=b=0$
(for a more detailed description see \cite{SartiPhD10}).
We conclude with some examples:
\begin{example}
For the projective space $\PR{4}{4}$ we have $q=85$ and $\ell=341$. The group generated by the 
Frobenius automorphism is the cyclic group $\Phi_{10}$. The order of $\Phi_{10}$ does not divide $q$,
so it cannot be a group of automorphisms of any of the dessins associated with $\PR{4}{4}$.
Nevertheless, the subgroup $\Phi_5 \subset \Phi_{10}$ with $5=m+1$ can be. The power $\sigma^2$ 
of the Frobenius automorphism 
can be chosen as a generator of $\Phi_5$. The action of $\Phi_5$ on the vertices 
of associated dessins 
is expressed by multiplication with $2^2$. Since $gcd(2^2-1,341)=1$, 
also the second of the two above conditions is satisfied.
This means that the elements of the difference set $D_5$ fixed by $\Phi_5$ are 
subdivided into $\Phi_5$-orbits of equal length: there are 17 orbits of legth five. 
The elements may, therefore, be ordered in a way compatible with the action of $\Phi_5$. If
this ordering is such that all differences of consecutive elements $(d_{i}-d_{i+1})$, 
$d_i, d_{i+1} \in D_5$ are prime to $\ell=341=31\cdot 11$, i.e. if it is Wada compatible, we may
construct a $( 85, 85, 341 )$-Wada dessin for which $\Phi_5$ is a group of automorphisms
acting freely on the 85 cells.
\end{example}
\begin{example}
For $\PR{6}{4}$ we have $q=1365$, $\ell=5461$.  The Frobenius automorphism generates the cyclic group
$\Phi_{14}$ with $gcd(14,1365)=7$. This means that the subgroup $\Phi_7 \subset \Phi_{14}$ 
has a suitable size to be a
group of automorphisms of possible Wada dessins associated with $\PR{6}{4}$. As in the 
above example, we may choose the power 
$\sigma^2$ of the Frobenius automorphism as a generator of $\Phi_7$. The action of $\Phi_7$
on the vertices of associated dessins is expressed by multiplication with powers $2^2$. Since
$gcd(2^2-1, 5461)=1$, also the second of the two above conditions
is satisfied and $\Phi_7$
divides the elements of the associated difference set $D_7$ into $195$ 
orbits of length $7$. We order the elements of the difference set
$D_7$ in a way compatible with the action of $\Phi_7$. If this ordering is also Wada compatible 
we may construct a $( 1365,1365,5461 )$-Wada dessin for which $\Phi_7$ is a group of automorphisms acting freely on the cells.
\end{example}
In general, it is not an easy task to construct difference sets with parameters $( v,k, \lambda )$
or, more specifically, with parameters $(\ell, q, \lambda )$ if they are associated with projective
spaces $\PR{m}{n}$. The difference sets used in our examples are known difference sets we took
from a difference set list in \cite{Baumert71}. In \cite{Baumert71} as well as in \cite{HugPip85} some
construction techniques are described. Of course, knowing that a difference set is a Frobenius 
difference set associated with a projective space $\PR{m}{n}$ and that its elements may be ordered
in a Frobenius compatible way may help in the construction (see e.g. the sections about
multipliers of difference sets in \cite[Section III.c]{Baumert71} and in \cite[Section 2.5]{HugPip85}).
In this case the structure of the difference set is, in fact, completely determined by the action of the 
cyclic group $\Phi_f$ or of a subgroup $\Phi_g$ of it, as we have seen in Section \ref{subsec_Frob_compat} 
and in the present section. Nevertheless, for projective spaces $\PR{m}{n}$ with arbitrary parameters
$m$ and $n$ it is difficult to predict when such orderings occur, as we have remarked above.
More research is needed in this direction.\par
The question about the existence of Wada compatible orderings of elements has been answerd by 
\cite{StWo01} and more recently by \cite{Goertz09} for difference sets associated with projective 
planes $\PR{2}{n}$. It is still an open and challenging question for projective spaces of 
higher dimension.

\section*{Acknowledgements}
This is a part of the PhD thesis of the author at the University of Frankfurt. She would like 
to warmly thank her advisors Jürgen Wolfart and Gareth Jones
for introducing her to the beautiful topic of 
dessins d'enfants and for many useful hints and discussions.\par
She would also like to thank Benjamin Mühlbauer, Alessandra Sarti and Ayberk Zeytin for 
carefully reading the first version of this paper.
Many thanks also to the anonymous reviewer of this paper for very useful comments. 

\bibliographystyle{alpha}
\bibliography{FrobeniusCompatibility}
\end{document}